\documentclass[12pt]{article}
\usepackage{amsmath,amssymb,amstext,dsfont,fancyvrb,float,fontenc,graphicx,subfigure,theorem,hyperref}
\usepackage[utf8]{inputenc}
\usepackage{multicol}

\usepackage{tikz,everypage}

\usepackage[letterpaper]{geometry}
\newfont{\footsc}{cmcsc10 at 8truept}
\newfont{\footbf}{cmbx10 at 8truept}
\newfont{\footrm}{cmr10 at 10truept}

\usepackage{fancyhdr}
\usepackage{relsize}
\usepackage{sectsty}
\allsectionsfont{\larger[-1]} 

\renewcommand\paragraph{\@startsection{paragraph}{4}{\z@}
                                    {2ex \@plus.5ex \@minus.2ex}
                                    {-1em}
                                    {\normalfont\normalsize\bfseries}}

\renewcommand\subparagraph{\@startsection{subparagraph}{5}{\parindent}
                                       {2ex \@plus.5ex \@minus .2ex}
                                       {-1em}
                                      {\normalfont\normalsize\bfseries}}

\newlength{\BiblioSpacing}
\renewenvironment{thebibliography}[1]{
\begin{oldthebibliography}{#1}
\setlength{\parskip}{\BiblioSpacing}
\setlength{\itemsep}{\BiblioSpacing}
}
{
\end{oldthebibliography}
}

\usepackage[strict]{changepage}
\def\abstractname{Abstract -}   
\def\abstract{\begin{adjustwidth}{1cm}{1cm} \par    \footnotesize \noindent {\bf \abstractname} 
\def\endabstract{ \end{adjustwidth} \smallskip }}

{\theorembodyfont{\itshape}\newtheorem{theorem}{Theorem}[section]}
{\theorembodyfont{\itshape}\newtheorem{proposition}[theorem]{Proposition}}
{\theorembodyfont{\itshape}\newtheorem{definition}[theorem]{Definition}}
{\theorembodyfont{\itshape}}
{\theorembodyfont{\itshape}\newtheorem{corollary}[theorem]{Corollary}}
{\theorembodyfont{\rm}}
{\theorembodyfont{\rm}\newtheorem{remark}[theorem]{Remark}}
{\theorembodyfont{\rm}}
{\theorembodyfont{\rm }}

\title{\Large\bf Permanents of $2\times 2$ Matrices Modulo $n$}

\author{\sc A. Bohra and A. Satyanarayana Reddy}

\def \N {{\mathbb{N}}}
\def \Z {{\mathbb{Z}}}

\begin{document}
\date{}
\maketitle
\thispagestyle{fancy}

\vskip 1.5em
\begin{abstract}
In this article we compute the number of invertible $2\times 2$ matrices with integer entries modulo $n$ whose permanents are congruent modulo $n$ to a given integer $x$.
\end{abstract}

\begin{MSC}
05B10; 15A15.
\end{MSC}
\section{Introduction and Preliminaries}\label{sec:intro}

Let us see why counting the matrices described in the abstract is a natural idea.

Let $\mathbb{Z}_n=\mathbb{Z}/n\mathbb{Z} =
\{0, 1, 2, \ldots, n-1\} $ denote the ring of integers modulo $n$. 
We will denote GCD$(a,b)$  by $(a,b)$. 

Lockhart and Wardlaw compute in \cite{lockhart} the number of square matrices with entries in $\mathbb{Z}_n$ with given determinant $x$. Trying to do the same after replacing the determinant by the permanent looks like a daunting task, so let us examine the case of $2\times 2$ matrices.

Let $M_2(\mathbb{Z}_n) $ denote the ring of all $2\times 2$ matrices with entries from $\mathbb{Z}_n$. 
\begin{definition}
Let 
$A = \begin{bmatrix}
                a & b \\
                c & d\\
        \end{bmatrix}  \in M_2(\mathbb{Z}_n)$. The {\em permanent} of $A,$ denoted $perm(A)$, is defined as 
    $$perm(A) = ad+bc \pmod{n}.$$
\end{definition}

If $n, x \in \mathbb{N}$, define 
$$D_n(x) = \{ A \in M_2(\mathbb{Z}_n) | det(A) \equiv x \pmod{n} \}.$$
The number of elements of $D_n(x)$, $|D_n(x)|$, was found in \cite{lockhart}. If we now define
$$F_n(x) = \{ A \in M_2(\mathbb{Z}_n) | perm(A) \equiv x \pmod{n} \},$$ 
and try to find $|F_n(x)|$, there is good news and bad news. The bad news is that it is easy to see that this will be the same number:
the sets $F_n(x) $ and $D_n(x)$ are in bijection via the map sending $\begin{bmatrix}
                                                      a & b \\
                                                      c & d
                                                      \end{bmatrix} \mapsto \begin{bmatrix}
                                                      a & -b \\
                                                      c & d
                                                      \end{bmatrix}.$
The good news is that $|F_n(x)| $ shares the properties of $|D_n(x)|$~\cite{lockhart}, such as : \begin{enumerate}
    \item Multiplicativity in $n$ : If $a,b \in \mathbb{N} \ $ with  $(a,b) = 1,$ then 
$$|F_{ab}(x)| = |F_{a}(x)| \times |F_{b}(x)|$$
    \item GCD invariance : If $n, x, y \in \mathbb{N}$, then 
$${|F_n(x)| = |F_n(y)|}  \mbox{ whenever }(n,x) = (n,y)$$
\end{enumerate}
For example, this allows us to conclude that 
 $$|M_2(\mathbb{Z}_n)| = \sum\limits_{d|n} \varphi \left(\frac{n}{d}\right)  |F_n(d)|.$$

 Denote by $GL_2(\mathbb{Z}_n)$ the group of units of $M_2(\mathbb{Z}_n) $. If we now set
$$G_n(x) = \{ A \in GL_2(\mathbb{Z}_n) | perm(A) \equiv x \pmod{n} \} \mbox{ and }
g_n(x)=|G_n(x)|,$$  
we will see that computing the values of $g_n(x)$ for every $n$ and $x$ is an interesting problem.

It is known that $|GL_2(\mathbb{Z}_p)| = (p^2 - p)(p^2 - 1)$. If $\psi :  GL_2(\mathbb{Z}_{p^k}) \rightarrow GL_2(\mathbb{Z}_{p})$ is the group homomorphism induced by the natural ring homomorphism $M_2(\mathbb{Z}_{p^k}) \rightarrow M_2(\mathbb{Z}_{p})$, then it is easy to check that  $|ker \psi| = p^{4(k-1)}.$ From the surjectivity of $\psi$, we have  
\begin{equation}\label{eq-one}
|GL_2(\mathbb{Z}_{p^k})| = p^{4(k-1)} |GL_2(\mathbb{Z}_{p})|=  p^{4(k-1)} (p^2 - p)(p^2 - 1)
\end{equation}
(Note that a similar formula holds in arbitrary dimension \cite{lockhart}.)

\section{The Function \texorpdfstring{$g_n(x)$}{gn(x)}}\label{sec:gn(x)}
We start this section  with  the following remark related to $G_n(0),$  which we will use frequently. 
\begin{remark}\label{rem:rp}
 Let $A = \begin{bmatrix}
a & b \\
c & d
\end{bmatrix}\in G_n(0).$ Then  $perm(A) \equiv 0 \pmod n$ and 
$(det(A),n) = (ad-bc,n) = 1.$ We claim that each of $a,b,c,d$ is relatively prime to $n.$ Suppose it were not true, without loss of generality  assume that $(a,n) = r > 1. $ Then there is a prime $p$ such that $p|r.$ As a consequence $p|a,$ now $p$  is prime and $p|ad+bc$ implies $p|b$ or $p|c$ which contradicts the fact that 
$(ad-bc,n) = 1.$  
\end{remark}
An immediate consequence of above remark is  that for an odd prime $p$ and $k\in  \mathbb{N}$ we have
\begin{equation}\label{eq-two}
g_{p^k}(0) = \phi(p^k)^3=(p^k-p^{k-1})^3
\end{equation}

The following result shows that  $g_n(x)$ is multiplicative in $n$, which is a direct consequence of the proof that  $|F_n(x)|$ is multiplicative in $n$ \cite{lockhart}. We are including the proof for the sake of completeness.
 With Proposition~\ref{pro:multg} and Proposition~\ref{pro:zero}  we can compute $g_n(x)$ for every $n$ and $x$. In the process we show that $g_{p^k}(x)$ assumes only two values. In particular, we show that $g_{p^k}(x)=\begin{cases}
                   g_{p^k}(0) & \mbox{if\;$p|x$},\\
                   g_{p^k}(1) & \mbox{otherwise}.
                                                                                                                                                                                                                                                                                                                                                                                                                                                                                                                                                                                            \end{cases}$
\begin{proposition}\label{pro:multg}
Let $a, b \in \mathbb{N} $ be such that $(a,b)=1.$ Then for every  $x\in \mathbb{Z}$ $$g_{ab}(x) = g_a(x) \times g_b(x).$$
\end{proposition}
\begin{proof} 
We make use of the bijection we already have from $F_{ab}(x) \to F_{a}(x) \times F_{b}(x) $ in \cite{lockhart} which takes $ [c_{ij}] \mapsto \left( [c_{ij} \ \pmod a] , [c_{ij}\pmod b] \right).$ Consider the same map from $G_{ab}(x) \to G_{a}(x) \times G_{b}(x).$ 
It is injective on the restricted domain since it is injective on the superset $F_{ab}(x).$ The range of the map is clearly contained in $G_{a}(x) \times G_{b}(x) $ since a unit (mod $ab$) is also a unit (mod $a$) and (mod $b$) when $(a , b) = 1.$ Now given $(R,S) \in G_{a}(x) \times G_{b}(x) $ we have a pre-image $T \in F_{ab}(x).$ We claim that this $T $ lies in $G_{ab}(x).$ For if it was not, then $det(T) $ would be a zero divisor in  $\mathbb{Z}_{ab}$, that would contradict $R $ and $S $ both being invertible. Thus the map is surjective on the restricted domain as well and hence a bijection from $G_{ab}(x) \to G_{a}(x) \times G_{b}(x).$ Thus $g_n(x) $ is multiplicative.
\end{proof}

The next result describes $g_n(x)$ when $n$ is a power of a prime.
\begin{proposition}\label{pro:zero}
Let $p \in \mathbb{N} $ be a prime number  and  $k \in \N.$ Then:\\
(i) for every $x\in \Z$ with $p| x$, $g_{p^k}(x)=g_{p^k}(0)$. Furthermore, if $p = 2, $ then $g_{2^k}(0) = 0.$\\
(ii) if $p\nmid x$, then $g_{p^k}(x)=g_{p^k}(1)$.\\
\end{proposition}
\begin{proof}
(i) We first construct a bijection from $G_{p^k}(0) $ to $G_{p^k}(p^i)$: 
 \begin{center}
    $\lambda : G_{p^k}(0) \rightarrow G_{p^k}(p^i)$ by 
$\begin{bmatrix}
        a & b \\
        c & d
\end{bmatrix} \mapsto 
\begin{bmatrix}
        a & b \\
        c & d + a^{-1}p^{i}
\end{bmatrix}.$
\end{center}
It is easy to check that the above map is injective. Given $A \in G_{p^k}(0), $ $\lambda(A) \in G_{p^k}(p^i).$ For gcd$(ad - bc, p^k) = 1 \implies  $  gcd $((ad - bc) + p^i, p^k) = $ gcd $(det(\lambda(A)), p^k) = 1.$ Furthermore, $\lambda $ is a surjective map as well. This is because given 
$B = \begin{bmatrix}
        a' & b' \\
        c' & d'
\end{bmatrix} \in G_{p^k}(p^i),$ we are guaranteed the existence of the multiplicative inverse $(a')^{-1} \in \mathbb{Z}_{p^k}$ (this can be proved in the same way as  Remark \ref{rem:rp}). Now consider
$C = \begin{bmatrix}
        a' & b' \\
        c' & d' - (a')^{-1}p^i
\end{bmatrix} \in G_{p^k}(0). $ Clearly, $\lambda(C) = B $ and thus \begin{center}
    $|G_{p^k}(0)| = |G_{p^k}(p^i)| \implies g_{p^k}(0) = g_{p^k}(p^i).$
\end{center}
Now, if $p\nmid m$,
define $h_{m} : G_{p^k}(p^i) \rightarrow G_{p^k}(mp^{i}) $ by  $\begin{bmatrix}
        a & b \\
        c & d
\end{bmatrix} \mapsto 
\begin{bmatrix}
        ma & mb \\
        c & d 
\end{bmatrix}.$ This is a bijection, since $(m,p^{k})=1$ there exists $t \in \Z_{p^k}$ such that $mt \equiv 1 \pmod {p^k}$, and the map  $h_t :G_{p^k}(mp^i)\to G_{p^k}(p^i)$  defined as $\begin{bmatrix}
       p & q\\
       r & s                                                                                                                                                                                                                                                 \end{bmatrix}\mapsto \begin{bmatrix}
        tp & tq\\
       r & s                                                                                                                                                                                                                                                 \end{bmatrix}$ is the inverse of $h_m.$
Thus \begin{center}
    $g_{p^k}(0) = g_{p^k}(p^i) = g_{p^k}(mp^i).$
\end{center}
Now consider the case when $p=2$ and let $A = \begin{bmatrix}
                a & b \\
                c & d\\
        \end{bmatrix} \in G_{2^k}(0)$, so $ad + bc \equiv 0 \pmod{2^k}$ and $2\nmid det(A) \equiv ad - bc$.
        Since $2 | ad + bc, $ both $ad $ and $bc $ have the same parity, 
so $2 | ad - bc $ as well, a contradiction. Thus
            $$g_{2^k}(0) = 0.$$
 (ii) We prove, more generally, that $g_n(k)=g_n(1)$ if $(k,n)=1.$ There exists $\ell\in \Z_n$ such that $k\ell\equiv 1 \pmod n.$ The map $h_k:G_n(1)\to G_n(k)$  defined by $\begin{bmatrix}
       a & b\\
       c & d                                                                                                                                                                                                                                                 \end{bmatrix}\mapsto \begin{bmatrix}
       ka & kb\\
       c & d                                                                                                                                                                                                                                                 \end{bmatrix}$,  has inverse  $h_\ell :G_n(k)\to G_n(1)$  defined by $\begin{bmatrix}
       p & q\\
       r & s                                                                                                                                                                                                                                                 \end{bmatrix}\mapsto \begin{bmatrix}
       \ell p & \ell q\\
       r & s                                                                                                                                                                                                                                                 \end{bmatrix}$.
\end{proof}

The following result follows immediately from Propositions \ref{pro:multg} and \ref{pro:zero}
\begin{corollary} Let $p_1,p_2,\ldots, p_r$ be distinct odd primes.\\
(i)  Let $n = p_{1}^{a_1}\cdots p_{r}^{a_r} \in \mathbb{N}.$  Then $g_n(x) $ takes $2^r$ possible values. \\
 (ii) Let $n = 2^{a}p_{1}^{a_1}\cdots p_{r}^{a_r} \in \mathbb{N},$ where  $a > 0.$ Then $g_n(x) $ takes $2^{r - 1} + 1 $ possible values. 
\end{corollary}

\begin{corollary}\label{thm:main}
 Let $n\in \N.$ Then $|GL_2(\mathbb{Z}_n)|=\sum\limits_{d|n}\varphi(\frac{n}{d})g_n(d).$
\end{corollary}
\begin{proof}
We have
$$|GL_2(\mathbb{Z}_n)|=g_n(1)+g_n(2)+\ldots+g_n(n),$$
so it is sufficient to prove that $g_n(a)=g_n(b)$ whenever $(a,n)=(b,n).$ If  $n=p_1^{a_1}p_2^{a_2}\cdots p_r^{a_r}$ and $(k,n)=d>1$, then $d= p_1^{b_1}p_2^{b_2}\cdots p_r^{b_r},$ where $0\le b_i\le a_i$, and from Proposition~\ref{pro:multg} we have 
$$g_n(d)=\prod_{p_i\mid d} g_{p_i^{a_i}}(d) \prod_{p_j\nmid d} g_{p_j^{a_j}}(d).$$
Now from Proposition~\ref{pro:zero} we have 
$$g_n(d)=\prod_{p_i\mid d} g_{p_i^{a_i}}(0) \prod_{p_j\nmid d} g_{p_j^{a_j}}(1)=g_n(k).$$
\end{proof}

\begin{corollary}
Let $p $ be an odd prime and $k \in \mathbb{N}.$ Then:\\
(i) $g_{p^k}(0)=(p^k-p^{k-1})^3$\\
(ii) $g_{p^k}(1) = p^{3(k-1)}(p-1)(p^2 + 1)$\\
(iii) $g_{2^k}(0)=0$\\
(iv)  $g_{2^k}(1) = 6 \times 8^{k-1}.$  
\end{corollary}
\begin{proof}
(i) is (\ref{eq-two}).\\
(ii) follows from (i), (\ref{eq-one}), Proposition~\ref{pro:zero}, and Corollary~\ref{thm:main}.\\
(iii) is in Proposition~\ref{pro:zero} (i).\\
(iv) follows from (iii), (\ref{eq-one}), Proposition~\ref{pro:zero}, and Corollary~\ref{thm:main}.
\end{proof}
\vspace{3mm}

We can now use the fact that $g_n(x)$ is multiplicative in $n$ and the previous corollary to find values of $g_n(x)$ for all $n, x\in\mathbb{N}$. We illustrate this process with a couple of examples.
\newpage

For the first example, $g_{75}(x)$ can take four values:
 \begin{multicols}{2}
\begin{enumerate}
    \item $g_{75}(0) = g_3(0)g_{25}(0) = 64,000$
    \item $g_{75}(1) = g_3(1)g_{25}(1) = 260,000$
    \item $g_{75}(3) = g_3(0)g_{25}(1) = 104,000$
    \item $g_{75}(5) = g_3(1)g_{25}(0) = 160,000.$
 \end{enumerate}
 \end{multicols}
Any other $g_{75}(x)$ is equal to one of these four, e.g.
$$g_{75}(14) = g_3(14)g_{25}(14) = g_3(1)g_{25}(1) = 260,000$$
and
$$g_{75}(35) = g_3(35)g_{25}(35) = g_3(1)g_{25}(0) = 160,000.$$

For the second example, the three possible values of $g_{24}(x) $ are : 
\begin{multicols}{2}
\begin{enumerate}
    \item $g_{24}(0) = g_8(0)g_3(0) = 0$
    \item $g_{24}(1) = g_8(1)g_3(1) = 7,680$
    \item $g_{24}(3) = g_8(1)g_3(0) = 3,072.$
\end{enumerate}
\end{multicols}
Any other $g_{24}(x)$ is equal to one of these three, e.g.
$$g_{24}(14) = g_8(14)g_{3}(14) = g_8(0)g_{3}(1) = 0\cdot g_{3}(1)= 0$$
and
$$g_{24}(35) = g_{24}(11) = g_8(11)g_{3}(11) = g_8(1)g_{3}(1) = 7,680.$$

{\footnotesize  
\medskip
\medskip
\vspace*{1mm} 
 
\noindent {\it Ayush Bohra}\\  
Department of Mathematics,\\
Shiv Nadar University,\\
India-201314.\\
E-mail: {\tt ab424@snu.edu.in }\\ \\  

\noindent {\it Arikatla Satyanarayana Reddy}\\  
Department of Mathematics,\\
Shiv Nadar University,\\
India-201314.\\
E-mail: {\tt satya.a@snu.edu.in}\\ \\

}


{\footnotesize
\begin{thebibliography}{00}


\bibitem{lockhart}
J.M. Lockhart, W.P. Wardlaw, Determinants of Matrices over the Integers Modulo $m$, {\it Math. Mag.}, {\bf 80} (2007), 207--214. Also available at the URL:\url{http://www.jstor.org/stable/27643029}

\end{thebibliography}}
\end{document}